\documentclass[11pt]{amsart}

\usepackage{amsrefs,amstext,amsmath,amsthm,amsfonts,amssymb,enumerate,amscd,cite}
\usepackage{mathrsfs}

\newtheorem{theorem}{Theorem}[section]

\newtheorem{proposition}[theorem]{Proposition}

\newtheorem{lemma}[theorem]{Lemma}

\newtheorem{definition}[theorem]{Definition}
\newtheorem{remark}[theorem]{Remark}
\newtheorem{remarks}[theorem]{Remarks}

\newcommand{\CC}{\ensuremath{\mathbb{C}}} 
\newcommand{\NN}{\ensuremath{\mathbb{N}}} 

\newcommand{\defeq}{:=} 
\newcommand{\integ}[3][]{\ensuremath{\int_{#1}#2\,d#3}} 
\newcommand{\id}{\mathrm{id}} 
\newcommand{\Ad}{\mathop{\mathrm{Ad}}} 
\newcommand{\inv}{^{-1}} 

\newcommand{\nrm}[1][\cdot]{\ensuremath{\Vert {#1} \Vert}} 
\newcommand{\dualp}[2]{\left\langle #1 , #2 \right\rangle} 

\newcommand{\HH}{\ensuremath{\mathcal{H}}} 
\newcommand{\ip}[2]{\left\langle #1 , #2 \right\rangle} 

\newcommand{\grp}{G} 
\newcommand{\lreg}{\lambda} 
\newcommand{\rreg}{\rho} 
\newcommand{\mul}{\mathrm{m}} 
\newcommand{\lrega}{\lambda} 
\newcommand{\dgrp}[1][\grp]{\hat{#1}} 
\newcommand{\Aut}{\mathop{\mathrm{Aut}}} 
\newcommand{\falg}[1][G]{\ensuremath{A(#1)}} 
\newcommand{\FSalg}[1][\grp]{B(#1)} 

\newcommand{\CA}{$C^*$-algebra}
\newcommand{\vNA}{von Neumann algebra}
\newcommand{\Bd}[1][\HH]{\ensuremath{\mathcal{B}(#1)}} 
\newcommand{\cpt}[1][\HH]{\mathcal{K}(#1)} 
\newcommand{\btens}{\mathbin{\overline{\otimes}}} 
\newcommand{\mintens}{\mathbin{\otimes_{\mathrm{min}}}} 
\newcommand{\vNal}{M} 
\newcommand{\vNall}{N} 

\newcommand{\module}{A} 
\newcommand{\vN}[1][\grp]{\ensuremath{\mathrm{vN}(#1)}} 
\newcommand{\Linf}[1][\grp]{L^\infty(#1)} 
\newcommand{\Ltwo}[1][\grp]{L^2(#1)} 
\newcommand{\vNcros}[3]{\ensuremath{#1 \rtimes_{#3} \! #2}} 
\newcommand{\vNcrs}{\vNcros{\vNal}{\grp}{\act}} 
\newcommand{\vNcocros}[3]{\ensuremath{#1 \rtimes_{#3} \! #2}} 
\newcommand{\vNcocrs}{\vNcocros{\vNal}{\grp}{\coact}} 

\newcommand{\act}{\alpha} 
\newcommand{\actG}{\alpha^{\grp}} 
\newcommand{\piact}[1][{\act}]{\pi_{#1}} 
\newcommand{\coact}{\delta} 
\newcommand{\coactG}{\delta^{\grp}} 
\newcommand{\picoact}[1][{\coact}]{\pi_{#1}} 
\newcommand{\dcoact}[1][\act]{\hat{#1}} 
\newcommand{\dact}[1][\coact]{\hat{#1}} 

\newcommand{\CB}[1][A]{\ensuremath{\mathcal{CB}(#1)}} 
\newcommand{\CBw}[1][\vNal]{\ensuremath{\mathcal{CB}_\sigma(#1)}} 
\newcommand{\CBwmod}[2][\vNal]{\ensuremath{\mathcal{CB}_{\sigma}^{#2}(#1)}} 

\title{Multipliers and Duality for Group Actions}
\author[A. McKee]{Andrew McKee}
\address{Department of Mathematical Sciences, Chalmers University of Technology and the University of Gothenburg, Gothenburg SE-412 96, Sweden}
\email{amckee240@qub.ac.uk}

\begin{document}

\bibliographystyle{plain}

\begin{abstract}
We define operator-valued Schur and Herz--Schur multipliers in terms of module actions, and show that the standard properties of these multipliers follow from well-known facts about these module actions and duality theory for group actions.
These results are applied to study the Herz--Schur multipliers of an abelian group acting on its Pontryagin dual: it is shown that a natural subset of these Herz--Schur multipliers can be identified with the classical Herz--Schur multipliers of the direct product of the group with its dual group.
\end{abstract}

\maketitle

\section{Introduction}
\label{sec:intro}


\noindent
Schur multipliers --- the scalar-valued functions on $\NN \times \NN$ for which the entrywise product maps $\Bd[\ell^2]$ into itself --- arose from Schur's work on the entrywise product of matrices in the early twentieth century.
Their importance was recognised by Grothendieck~\cite{Gr53} (see also Pisier~\cite[Chapter 5]{Pis96}), who used them to formulate his fundamental theorem.
These classical Schur multipliers have been extended in several directions; see, for example, \cite[Chapter 5]{Pis96}. 

Herz--Schur multipliers, or completely bounded multipliers of the Fourier algebra of a group, originate in work of Herz~\cite{Her74} where they were viewed as a generalisation of the Fourier--Stieltjes transform. 
They have proved useful in the study of approximation properties of operator algebras associated to groups; this was first made explicit by De~Canni\`ere and Haagerup~\cite{dCH85}, and has since been exploited by many other authors (see \cite[Chapter 12]{BrO08} for further references). 
This utility has driven the development of several classes of Herz--Schur multipliers, for example the radial multipliers which first appeared in \cite{Haa78}.

Bo\.zejko and Fendler~\cite{BF84} linked these two notions, using unpublished work of Gilbert (see also Jolissaint~\cite{Jo92}) to give a `transference' theorem, showing that every Herz--Schur multiplier of $\grp$ gives rise to a Schur multiplier acting on $\Bd[\Ltwo]$.
Moreover, one can characterise the Herz--Schur multipliers as those Schur multipliers which are invariant, in the sense that they commute with conjugation by the right regular representation of $\grp$.
We regard the transference and characterisation results as important goals of the generalised theory presented here.


The importance of the theory of multipliers has led to several authors introducing operator-valued versions of Schur multipliers \cites{BGB18,MTT18} and Herz--Schur multipliers \cites{AD87,BC15,DR12,MTT18}.
In particular our work with Todorov and Turowska~\cite{MTT18} develops and studies \CA -valued versions of Schur and Herz--Schur multipliers, including both transference and characterisation theorems. The present work arose from an attempt to distill the essential features of some of the proofs given in that paper.

Aspects of the theory of Schur and Herz--Schur multipliers have also been generalised to quantum groups. For example, Junge--Neufang--Ruan~\cite{JNR09} give a transference theorem in the setting of locally compact quantum groups, and Brannan~\cite{Bra17} uses similar ideas when discussing approximation properties of quantum groups. 

This paper serves two purposes: firstly we show how to obtain the main results of \cite{MTT18} in the \vNA\ setting, and secondly we show how the definitions and important properties of (operator-valued) Schur and Herz--Schur multipliers can be obtained from basic properties of group and module actions on operator algebras.
More specifically, after preliminaries in Section~\ref{sec:prelimins}, in Section~\ref{sec:Schurmultsmodule} we define Schur multipliers as completely bounded maps commuting with a particular module action, and obtain a dilation-type characterisation of these multipliers in Theorem~\ref{th:dilationSchurmults}.

Section~\ref{sec:HSchmultsnew} begins with the definition of a Herz--Schur multiplier of a group action, so that the classical Herz--Schur multipliers are the Herz--Schur multipliers of the trivial action of the group on $\CC$.
We then prove the main results of the paper: Proposition~\ref{pr:HSmultsareSmultsofdualcoact} is a version of transference for our multipliers, identifying the Herz--Schur multipliers of a group action with certain Schur multipliers associated to the dual coaction, and a characterisation of the Schur multipliers which arise in this way in Theorem~\ref{th:HSmultsinSchmults}.

In Section~\ref{sec:abeliangrooups} we focus on abelian groups. 
When $\grp$ is abelian the algebra $\Bd[\Ltwo]$ is the crossed product formed by an action of either $\grp$ or the dual group $\dgrp$, and in Theorem~\ref{th:FSalgmodulemultipliers} we characterise the maps on $\Bd[\Ltwo]$ which are Herz--Schur multipliers of both actions simultaneously as the Herz--Schur multipliers of $\grp \times \dgrp$.


Finally, we note that preliminary investigations have recovered some of the results of this paper for Kac algebras. This will be explored in a future work.

\section{Preliminaries}
\label{sec:prelimins}

\noindent
Throughout, $\grp$ denotes a locally compact group, and $\vNal$ a \vNA\ acting on the Hilbert space $\HH_M$.
The normal spatial tensor product of \vNA s will be denoted by $\btens$. The unit of $\vNal$ will be written $1_\vNal$, $\id$ denotes the identity representation of a \vNA, and $I_\HH$ the identity operator on the Hilbert space $\HH$.

We follow Nakagami--Takesaki~\cite{NaTa79} (except that we use the left group von~Neumann algebra). 
An \emph{action} of $\grp$ on $\vNal$ is a homomorphism $\act : \grp \to \Aut(\vNal)$, continuous in the point-weak* topology. Equivalently, there is a normal $*$-isomorphism $\piact : \vNal \to \vNal \btens \Linf$ satisfying 
\begin{equation*}\label{eq:actionidentity}
    (\piact \otimes \id) \circ \piact = (\id \otimes \piact[\actG]) \circ \piact .
\end{equation*}
Here $\actG$ denotes the action of $\grp$ on $\Linf$, so that 
\begin{equation*}\label{eq:coproductonLinfty}
    \piact[\actG] : \Linf \to \Linf \btens \Linf ;\ \piact[\actG](f)(s,t) \defeq f(st) , \quad f \in \Linf ,\ s,t \in \grp ,
\end{equation*}
which is the coproduct on $\Linf$.
Given an action $\act$ the corresponding isomorphism $\piact$ is defined by
\[
    \piact(a) \xi(s) \defeq \act_s\inv (a) \xi(s) , \quad a \in \vNal ,\ s \in \grp ,\ \xi \in \Ltwo[\grp , \HH_\vNal] .
\]
The \emph{crossed product} associated to the action $\act$, denoted $\vNcrs$, is the \vNA\ on $\HH_\vNal \otimes \Ltwo$ generated by $\piact(\vNal)$ and $\CC \otimes \vN$. Note that $\vN$ is the crossed product formed by the trivial action of $\grp$ on $\CC$.

The definitions for a coaction of $\grp$ are identical to the above, except that the roles of $\vN$ and $\Linf$ are exchanged: a \emph{coaction} $\coact$ of $\grp$ on $\vNal$ is a normal $*$-isomorphism $\picoact : \vNal \to \vNal \btens \vN$ satisfying
\begin{equation*}\label{eq:coactionidentity}
    (\picoact \otimes \id) \circ \picoact = (\id \otimes \picoact[\coactG]) \circ \picoact .
\end{equation*}
Here $\coactG$ denotes the coaction of $\grp$ on itself, so that 
\begin{equation*}\label{eq:coproductonvN}
    \picoact[\coactG] : \vN \to \vN \btens \vN ;\ \picoact[\coactG](\lreg_r) \defeq \lreg_r \otimes \lreg_r ,\ r \in \grp ,
\end{equation*}
which is the coproduct on $\vN$.
The \emph{crossed product} associated to the coaction $\coact$, denoted $\vNcocrs$, is the \vNA\ on $\HH_\vNal \otimes \Ltwo$ generated by $\picoact(\vNal)$ and $\CC \otimes \Linf$. Note that the crossed product formed by the trivial coaction of $\grp$ on $\CC$ is $\Linf$.
When $\grp$ is abelian $\vN$ can be identified with $\Linf[\dgrp]$, so in this case a coaction of $\grp$ is an action of $\dgrp$.

Given an action $\act$ of $\grp$ on $\vNal$ there is a \emph{dual coaction} $\dcoact$ of $\grp$ on $\vNcrs$, given by
\begin{equation*}\label{eq:dualcoact}
    \picoact[\dcoact] \big( \piact(a) \lrega_r \big) \defeq \piact(a) \lrega_r \otimes \lreg_r , \quad a \in \vNal ,\ r \in \grp .
\end{equation*}
Similarly, given a coaction $\coact$ there is a \emph{dual action} $\dact$ of $\grp$ on $\vNcocrs$.
The Takai duality theorem for abelian groups can be generalised to this setting, and gives isomorphisms
\begin{align*}\label{eq:dualitytheorems}
    \vNcocros{\big( \vNcrs \big)}{\grp}{\dcoact} &\cong \vNal \btens \Bd[\Ltwo] \quad \text{and} \\
    \vNcros{\big( \vNcocrs \big)}{\grp}{\dact} &\cong \vNal \btens \Bd[\Ltwo] .
\end{align*}
We will denote the first of these isomorphisms by $\Phi$; 
it is given on the generators of $\vNal \btens \Bd[\Ltwo]$ by
\begin{equation}\label{eq:secdualactidentif}
    \Phi \big( \piact(a) \big) = \piact(a) \otimes I_{\Ltwo} , \quad \Phi(I_\HH \otimes \lreg_r) = I_\HH \otimes \lreg_r \otimes \lreg_r , \quad \Phi(I_\HH \otimes \phi) = I_\HH \otimes I_{\Ltwo} \otimes \phi ,
\end{equation}
where $a \in \vNal , \ r \in \grp , \ \phi \in \Linf$ \cite[page 8]{NaTa79}.
Under $\Phi$ the second dual action $\dact[\dcoact]$ of $\grp$ on $\vNcocros{\big( \vNcrs \big)}{\grp}{\dcoact}$ is identified with the action $\act \otimes \Ad \rreg$ on $\vNal \btens \Bd[\Ltwo]$, where $\rreg$ is the right regular representation of $\grp$.

We use the basic theory of operator spaces and completely bounded maps, as found in \cite{EfRu00} for example, without comment.
The space of completely bounded, weak*-continuous maps on a \vNA\ $\vNal$ will be written $\CBw$; if $\vNal$ is also a bimodule over $\module$ then the completely bounded, weak*-continuous $\module$-bimodule maps on $\vNal$ will be denoted by $\CBwmod{\module}$.

\section{Schur multipliers}
\label{sec:Schurmultsmodule}

\noindent
In this section we define generalised Schur multipliers.
Throughout $X = (X ,\mu)$ denotes a standard measure space for which the underlying topology is locally compact. 

\begin{definition}\label{de:Schurmultmodule}
A \emph{Schur $X$-multiplier of $\vNal$} is a completely bounded, weak*-continuous, $\Linf[X]$-bimodule map on $\vNal \btens \Bd[\Ltwo[X]]$.
Given a Banach algebra $\module$ such that $\vNal$ is an $\module$-(bi)module, equip $\vNal \btens \Bd[\Ltwo[X]]$ with the natural $\module$-(bi)module structure.
A \emph{Schur $X$-multiplier of $\vNal$ with respect to $\module$} is a Schur $X$-multiplier of $\vNal$ which is also an $\module$-(bi)module map.
\end{definition}


\begin{remarks}\label{re:Schurmultsbimod}
{\rm 
\begin{enumerate}[i.]
    \item When $\vNal = \CC$ the Schur multipliers defined above are the classical Schur multipliers. 
In this case we need only require boundedness of the $\Linf[X]$-bimodule map, complete boundedness follows automatically (see \textit{e.g.}\ \cite[Section 2]{PSS00}). 
    \item More generally, if $\vNall$ is matricially norming for $\vNal$ then any bounded map which is an $\vNall \btens \Linf[X]$-bimodule map is a Schur $X$-multiplier of $\vNal$ with respect to $\vNall$, since such a map is automatically completely bounded \cite{PSS00}. 
    \item Choosing $\HH_M = \Ltwo[Y]$, $\vNal = \Bd[\Ltwo[Y]]$ and $\module = \Linf[Y]$, with $Y$ a standard measure space with locally compact topology, 
the definition above becomes the completely bounded $\Linf[X \times Y]$-bimodule maps on $\Bd[\Ltwo[X \times Y]]$, {\it i.e.}\ the classical Schur multipliers on $\Bd[\Ltwo[X \times Y]]$.
    \item It is clear that a classical Schur $X$-multiplier defines a Schur $X$-multiplier of $\vNal$, and that the Schur multipliers of $\vNal$ of this form are module maps for any module structure on $\vNal$.
\end{enumerate}
}
\end{remarks}

%

Recall from \cite{MTT18} that given $k \in \Ltwo[X \times X] \odot \vNal$ one can associate a bounded operator $T_k$ by
\begin{equation}\label{eq:Tkdef}
    T_k : \Ltwo[X , \HH_\vNal] \to \Ltwo[X , \HH_\vNal] ;\ T_k \xi(y) \defeq \integ[X]{k(y,x) \xi(x)}{x} ,
\end{equation}
and that such operators are norm-dense in $\vNal \mintens \cpt[\Ltwo[X]]$.
In \cite[Theorem 2.6]{MTT18} we showed that the Schur multipliers defined there correspond to certain symbols $\varphi : X \times X \to \CB[\vNal]$ {\it via}
\[
    S_\varphi(T_k) \defeq T_{\varphi \cdot k} \quad \text{where $\varphi \cdot k (x,y) \defeq \varphi(y,x) \big( k(x,y) \big)$.}
\] 
In this paper we have defined Schur $X$-multipliers of $\vNal$, which act on $\vNal \btens \Bd[\Ltwo[X]]$; 
in the next result, which is based on \cite[Theorem 2.6]{MTT18}, we show that our definition of a Schur multiplier $S$ determines how $S$ acts on the operators $T_k$ defined above, and use this to associate a symbol to $S$.

\begin{theorem}\label{th:dilationSchurmults}
Let $\vNal$ be a \vNA\ on the separable Hilbert space $\HH_\vNal$. 
The following are equivalent:
\begin{enumerate}[i.]
    \item $S$ is a Schur $X$-multiplier of $\vNal$;
    \item there exists a bounded function $\varphi : X \times X \to \CBw$, of the form
        \[
            \varphi(x,y)(a) = W(y)^* \rho(a) V(x) , \quad x,y \in X,\ a \in \vNal ,
        \]
        with $\rho$ a normal representation of $\vNal$ and $V,W \in \Linf[X , \Bd[\HH_M , \HH_\rho]]$, such that $S = S_\varphi$.
\end{enumerate}
Moreover, if $\vNal$ is an $\module$-(bi)module then $S$ is an $\module$-(bi)module map if and only if $\varphi(x,y)$ is an $\module$-(bi)module map for almost all $x,y \in X$.
\end{theorem}
\begin{proof}
(i)$\implies$(ii)
Write $S = W_0^* \theta( \cdot ) V_0$, where $\theta$ is a normal representation of $\vNal \btens \Bd[\Ltwo[X]]$ on the Hilbert space $\HH_\theta$ and $V_0,W_0 \in \Bd[\HH_\vNal \otimes \Ltwo[X] , \HH_\theta]$.
The map
\[
    \theta_0(T) \defeq \theta(1_\vNal \otimes T) , \quad T \in \Bd[\Ltwo[X]] ,
\]
defines a normal representation of $\Bd[\Ltwo[X]]$ on $\HH_\theta$.
As is well known, this implies that we can write $\HH_\theta = \HH_\rho \otimes \Ltwo[X]$ for another Hilbert space $\HH_\rho$ and identify $\theta_0(T)$ with $I_{\HH_\rho} \otimes T$.
Since
\begin{gather*}
    \theta (a \otimes T) = \theta (a \otimes I_{L^2(X)}) \theta (I_{\HH_\rho} \otimes T) = \theta (a \otimes I_{L^2(X)}) \theta_0(T) \quad \text{ and } \\
        \theta (a \otimes T) = \theta_0(T) \theta (a \otimes I_{L^2(X)})
\end{gather*}
for all $a \in \vNal$ and $T \in \Bd[\Ltwo[X]]$ we have that $\theta (a \otimes I)$ commutes with $\CC \otimes \Bd[\Ltwo[X]]$, so we obtain a representation $\rho$ of $\vNal$ on $\HH_\rho$ such that $\theta (a \otimes T) = \rho(a) \otimes T$, acting on $\HH_\rho \otimes \Ltwo[X]$.
We now have that $S = W_0^* (\rho \otimes \id) (\cdot) V_0$ (identifying the ranges of $V_0$ and $W_0$ with $\HH_\rho \otimes \Ltwo[X]$).
Arguing as in \cite[Theorem 2.6]{MTT18}, using that $S$ commutes with $\CC \otimes \Linf[X]$, we can find projections $P,Q$ so that $V \defeq P V_0$ and $W \defeq Q W_0$ commute with $\CC \otimes \Linf[X]$, so by Takesaki~\cite[Theorem 7.10]{Tak02} $V,W \in L^\infty(X, \Bd[\HH_\vNal , \HH_\rho])$.
We can now conclude $S = S_\varphi$ with
\[
    \varphi(x,y)(a) = W(y)^* \rho(a) V(x) 
\]
as in \cite[Theorem 2.6]{MTT18}. 

(ii)$\implies$(i) It is clear that if $S = S_\varphi$ with $\varphi$ as in (ii) then $S = W^* (\rho \otimes \id)( \cdot ) V$, so $S$ is completely bounded.
Since $V,W \in L^\infty(X, \Bd[\HH_\vNal , \HH_\rho])$ it follows that $S$ is an $\Linf[X]$-bimodule map, so $S$ is a Schur $X$-multiplier of $\vNal$.

To show that $\varphi(x,y)$ is an $\module$-(bi)module map when $S$ is take $a \in \vNal,\ b \in \module ,\ k \in \Ltwo[X \times X]$. 
Then
\begin{equation*}\label{eq:SmodulephibimoduleI}
    (b \otimes \id) \big( S(a \otimes T_k) \big) = (b \otimes \id) \big( S(T_{a \otimes k} \big) = (b \otimes \id) T_{\varphi \cdot (a \otimes k)} = T_{b \cdot (\varphi \cdot (a \otimes k))} ,
\end{equation*}
and
\begin{equation*}\label{eq:SmodulephibimoduleII}
    \big( S(b \cdot a \otimes T_k ) \big) = T_{\varphi \cdot (b \cdot a \otimes k)} .
\end{equation*}
Since $S$ is a module map the last two displays are equal, which implies
\[
    k(y,x) \big( b \cdot \varphi(x,y)(a) \big) = k(y,x) \varphi(x,y)(b \cdot a) ,
\]
and using the fact that they hold for all $k \in \Ltwo[X \times X] ,\ x,y \in X$ we conclude that $\varphi(x,y)(b \cdot a) = b \cdot \varphi(x,y)(a)$ for all $a \in \vNal$ and $b \in \module$.
A similar calculation shows that $\varphi$ respects the right module action when $S$ does; thus $\varphi (x,y)$ is a completely bounded $\module$-(bi)module map on $\vNal$.
The converse follows similarly.
\end{proof}

\begin{remarks}\label{re:Schurmultdef}
{\rm 
\begin{enumerate}[i.]
    \item The above theorem reduces to a well-known characterisation of classical Schur multipliers when $\vNal = \CC$ \cites{Haaun,KaPa05}. 
    \item If $\act$ is an action of $\grp$ on $\vNal$ then the crossed product by the dual coaction $\dcoact$ is identified with $\vNal \btens \Bd[\Ltwo]$, and Theorem~\ref{th:dilationSchurmults} above identifies Schur multipliers on this space with functions $\grp \times \grp \to \CBw$. In the next section we will define Herz--Schur multipliers of $\act$ and identify them with a certain subspace of these Schur multipliers on $\grp \times \grp$.
\end{enumerate}
}
\end{remarks}

\section{Herz--Schur Multipliers}
\label{sec:HSchmultsnew}

\noindent
We are now going to define Herz--Schur multipliers for a group action on a \vNA .
Throughout, $\grp$ denotes a second-countable locally compact group. 

\begin{definition}\label{de:HSmultaction}
Let $\act$ be an action of $\grp$ on $\vNal$.
We say that a map $S : \vNcrs \to \vNcrs$ is a \emph{Herz--Schur multiplier of $\act$} if $S$ is completely bounded, weak*-continuous, and
\begin{equation}\label{eq:HerzSchurmultcondition}
    \picoact[\dcoact] \circ S = (S \otimes \id) \circ \picoact[\dcoact] .
\end{equation}
\end{definition}

\noindent
We will refer to a map $S$ satisfying condition (\ref{eq:HerzSchurmultcondition}) by writing ``$S$ commutes with $\dcoact$''.

\begin{remarks}\label{re:HSmultnotes}
{\rm 
\begin{enumerate}[i.]
    \item Condition (\ref{eq:HerzSchurmultcondition}) is the same as the condition which defines a Fourier multiplier of a locally compact quantum group (see \textit{e.g.\ } Brannnan~\cite[Proposition 4.5]{Bra17}). See also the definition of a (right) covariant map by Junge--Neufang--Ruan~\cite[pg 391]{JNR09}.
    
    \item In particular, it is straightforward to show that $T : \vN \to \vN$ defines a classical Herz--Schur multiplier if and only if $T_*$ is a completely bounded map on $\falg$ such that
\begin{equation}\label{eq:productconditionHSmultFalg}
    T_*(uv) = T_*(u) v , \quad u,v \in \falg .
\end{equation}
If $\act$ is the trivial action of $\grp$ on $\CC$ then $\dcoact = \coactG$, which induces the product on $\falg$.
If $T$ satisfies (\ref{eq:productconditionHSmultFalg}) then, for $x \in \vN ,\ u,v \in \falg$, the calculation
\[
    \dualp{\picoact[\coactG] \circ T(x)}{u \otimes v} = \dualp{T(x)}{uv} = \dualp{x}{T_*(u) v} = \dualp{\picoact[\coactG](x)}{T_*(u) \otimes v}  
\]
shows that $T$ satisfies (\ref{eq:HerzSchurmultcondition}). A similar calculation shows (\ref{eq:HerzSchurmultcondition}) implies (\ref{eq:productconditionHSmultFalg}).

    \item Observe that $\vNcrs$ carries an $\falg$-module structure: for $u \in \falg$ define 
\[
    u * x \defeq (\id \otimes u) \pi_{\dcoact} ( x ) , \quad x \in \vNcrs ,
\]
so that $u * (\pi_\act(a) \lrega_r) = u(r) \pi_\act(a) \lrega_r$.
It is easy to see that Definition~\ref{de:HSmultaction} is equivalent to requiring that $S$ commutes with this module action.

    \item Given a Herz--Schur multiplier of $\act$, say $S$, equation (\ref{eq:HerzSchurmultcondition}) and (iii) above imply that $S(\piact(a) \lrega_r) \in \piact(\vNal) \lrega_r$, so there is some $a_{S,r} \in \vNal$ with $S(\piact(a) \lrega_r) = \piact(a_{S,r}) \lrega_r$. 
Setting $F(r)(a) \defeq a_{S,r}$ we obtain a function $F$ on $\grp$ such that $F(r)$ is a linear map on $\vNal$ for each $r \in \grp$. Moreover, since $S$ is completely bounded and weak*-continuous $F(r)$ must be so too. This shows that for every Herz--Schur multiplier of $\alpha$ $S$ there is a \emph{symbol} $F : \grp \to \CBw$ such that
\begin{equation*}\label{eq:HSmultactsonfibres}
    S \big( \piact(a) \lrega_r \big) = \piact \big( F(r)(a) \big) \lrega_r , \quad a \in \vNal ,\ r \in \grp .
\end{equation*} 

    \item Suppose that $v : \grp \to \CC$ is a classical Herz--Schur multiplier of $\grp$. For any action $\act$ of $\grp$ on $\vNal$ we can extend $S_v$ to a completely bounded, weak*-continuous map on $\vNcrs$ by
\begin{equation*}
    S_v \big( \piact(a) \lrega_r \big) = v(r) \piact(a) \lrega_r , \quad a \in \vNal,\ r \in \grp .
\end{equation*}
It is easily checked that $S_v$ commutes with $\dcoact$, so that $S_v$ is a Herz--Schur multiplier of $\act$.
    
    \item Let $\grp$ be abelian, and consider the canonical action of $\grp$ on $\Linf = \vN[\dgrp]$. In \cite[Section 6]{MTT18} we showed that every element of $\FSalg \odot \FSalg[\dgrp]$ is a Herz--Schur multiplier of this action; moreover, by symmetry, each such multiplier is also a Herz--Schur multiplier of $\dgrp$ on $\Linf[\dgrp] = \vN$, and the multipliers of this form are $\falg$ module maps on $\vN$ (and $\falg[\dgrp]$ module maps on $\Linf$). We will study these multipliers further below.                                                                                                                                                                                                                                                                                                                                                                                                                                             
\end{enumerate}
}
\end{remarks}

If $\vNal$ has a (left) module structure over $\module$ we can introduce an $\module$-module structure on $\vNcrs$ by 
\begin{equation}\label{eq:modactionliftedtocross}
    b \cdot \piact(a) \lrega_r \defeq \piact (b \cdot a) \lrega_r , \quad b \in \module,\ a \in \vNal ,\ r \in \grp .
\end{equation}
It is easy to check that under the additional assumption 
\begin{equation}\label{eq:moduleactioninteractswithaction}
    b \cdot \act_r(a) = \act_r(b \cdot a) , \quad r \in \grp ,\ a \in \vNal ,\ b \in \module
\end{equation} 
this module action is the one induced on $\vNcrs$ by the canonical module action of $\module$ on $\vNal \btens \Bd[\Ltwo]$.
If $S$ is a Herz--Schur multiplier with symbol $F : \grp \to \CBwmod{\module}$ then $S$ is also an $\module$-module map, since
\[
\begin{split}
    S \big( b \cdot \piact(a) \lrega_r \big) = S \big( \piact(b \cdot a) \lrega_r \big) 
        &= \piact \big( b \cdot F(r)(a) \big) \lrega_r 
        = b \cdot \big( S ( \piact(a) \lrega_r ) \big) .
\end{split}
\] 

Recall that for an action $\act$ of $\grp$ on $\vNal$ the crossed product by the dual coaction $\dcoact$ of $\grp$ on $\vNcrs$ can be identified with $\vNal \btens \Bd[\Ltwo]$.
Given a map $R : \vNcrs \to \vNcrs$ we define a map $\overline{R}$ on $\vNal \btens \Bd[\Ltwo]$ by $\overline{R} \defeq \Phi\inv \circ (R \otimes \id) \circ \Phi$, where $\Phi$ is the isomorphism (\ref{eq:secdualactidentif}). 
Observe that $\overline{R}$ is completely bounded (resp.\ completely positive) if $R$ is.
In the remainder of this section we explain how Herz--Schur multipliers of $\act$ interact with the Schur multipliers of $\vNal \btens \Bd[\Ltwo]$.

\begin{lemma}\label{le:useofDiracfam}
Let $\act$ be an action of $\grp$ on $\vNal$. Fix $a \in \vNal,\ r \in \grp,\ \phi \in \Linf$ and suppose $(u_i)_i$ is a net of positive, compactly supported functions with $\nrm[u_i]_1 = 1$ whose support shrinks to $\{ r \}$. 
\begin{enumerate}[i.]
    \item The kernels $k_i(s,t) \defeq u_i(st\inv) \act_{s\inv}(a)$ satisfy $T_{k_i} \stackrel{w^*}{\to} \piact(a) \lrega_r$.
    \item The kernels $h_i(s,t) \defeq u_i(st\inv) (\act_{r\inv}(a) \otimes \actG_{r\inv}(\phi))$ satisfy $T_{h_i} \stackrel{w^*}{\to} a \otimes \phi \lreg_r$.
\end{enumerate}
\end{lemma}
\begin{proof}
Routine calculations show $\ip{T_{k_i} \xi}{\eta} \to \ip{\piact(a) \lrega_r \xi}{\eta}$ and $\ip{T_{h_i} \xi}{\eta} \to \ip{(a \otimes \phi\lreg_r) \xi}{\eta}$ for all $\xi,\eta \in \Ltwo[\grp , \HH]$. 
The conclusion follows because the weak* topology coincides with the WOT on bounded sets.
\end{proof}

\begin{lemma}\label{le:SchurmultgivesHSmultofdual}
Let $S : \vNal \btens \Bd[\Ltwo] \to \vNal \btens \Bd[\Ltwo]$ be a Schur multiplier, $\tau$ the trivial action of $\grp$ on $\vNal$, $\beta \defeq \tau \otimes \actG$ and $\Psi : \vNcros{(\vNal \btens \Linf)}{\grp}{\beta} \to \vNal \btens \Bd[\Ltwo]$ the canonical isomorphism.
Then $\tilde{S} \defeq \Psi\inv \circ S \circ \Psi$ is a Herz--Schur multiplier of $\beta$, \textit{i.e.}\ $\piact[{\dcoact[\beta]}] \circ \tilde{S} = (\tilde{S} \otimes \id) \circ \piact[{\dcoact[\beta]}]$.
\end{lemma}
\begin{proof}
Let $\varphi$ be the symbol of $S$, obtained in Theorem~\ref{th:dilationSchurmults}.
It is straightforward to check, using Lemma~\ref{le:useofDiracfam}, that for $r \in \grp$ we have $S(a \otimes \phi\lreg_r) = \varphi_r(a \otimes \phi) \lreg_r$, where $\varphi_r : \grp \to \CBw[\vNal \btens \Linf]$ is given by $\varphi_r(s)(x) \defeq \varphi(s,r\inv s)(\beta_{r\inv}(x))$.
Now we calculate, for $a \in \vNal$ and $\phi \in \Linf$, 
\[
\begin{split}
    \piact[{\dcoact[{\beta}]}] \circ \tilde{S} \big( \piact[\beta](a \otimes \phi) \lrega_r \big) &= \piact[{\dcoact[{\beta}]}] \circ \Psi\inv \big( \varphi_r(a \otimes \phi) \lrega_r \big) = \piact[{\dcoact[{\beta}]}] \Big( \piact[\beta] \big( \varphi_r(a \otimes \phi) \big) \lrega_r \Big) \\
        &= \piact[\beta] \big(\varphi_r(a \otimes \phi) \big) \lrega_r \otimes \lreg_r = \tilde{S} \big( \piact[\beta](a \otimes \phi) \lrega_r \big) \otimes \lreg_r \\
        &= (\tilde{S} \otimes \id) \circ \piact[{\dcoact[{\beta}]}] \big( \piact[\beta](a \otimes \phi) \lrega_r \big) ,
\end{split}
\]
which proves the claim.
\end{proof}

First we have a version of the transference theorem (see also \cites{BF84,MTT18,JNR09}). 

\begin{proposition}\label{pr:HSmultsareSmultsofdualcoact}
Let $\act$ be an action of $\grp$ on $\vNal$ and $S$ a Herz--Schur multiplier of $\act$ with symbol $F : \grp \to \CBw$.
Then $\overline{S}$ is a Schur multiplier of $\dcoact$ with symbol $\varphi(s,t)(a) = \act_{t\inv} (F(t s\inv)(\act_t (a)))$.
Moreover, if $\vNal$ has an $\module$-module structure satisfying (\ref{eq:moduleactioninteractswithaction}) and $F(r)$ is an $\module$-module map for all $r \in \grp$ then $\varphi(s,t)$ is an $\module$-module map for all $s,t \in \grp$, so $\overline{S}$ is also an $\module$-module map.
\end{proposition}
\begin{proof}
Let $S$ be a Herz--Schur multiplier of $\act$. 
For $a \in \vNal$, $r \in \grp$ and $\phi \in \Linf$ we have
\begin{equation*}\label{eq:calcHSisSchuondualcoact}
\begin{split}
    \overline{S} \Big( \big( \piact(a) \lrega_r \big) (I_{\HH_\vNal} \otimes \phi) \Big) &= \Phi\inv \circ (S \otimes \id ) \big( \piact(a) \lrega_r \otimes \lreg_r \phi) \\
    &= \Phi\inv \Big( \piact \big( F(r)(a) \big) \lrega_r \otimes \lreg_r \phi \Big) \\
    &= \Big( \overline{S} \big( \piact(a) \lrega_r \big) \Big) (I_{\HH_\vNal} \otimes \phi) ;
\end{split}
\end{equation*}
similarly $\overline{S}$ commutes with left multiplication by $\Linf[\grp]$.
That $\overline{S}$ is a Schur multiplier follows by linearity and weak*-continuity.

To calculate the symbol $\varphi$ associated to the Schur multiplier $\overline{S}$ fix $a \in \vNal$ and $r \in \grp$.
For $k \in \Ltwo[\grp \times \grp , \vNal]$ we define $k^r : \grp \to \vNal$ by $k^r(p) \defeq k(p , r\inv p)$. 
Let $(u_i)_{i \in I}$ and $(k_i)_{i \in I}$ be as in Lemma~\ref{le:useofDiracfam}.
Similarly one checks that $(k^r_i)_{i \in I}$ converge to $\piact(a)$ in the weak* topology of $\Linf[\grp , \vNal]$.
Since $\overline{S}(T_{k_i}) \to \piact(F(r) (a)) \lrega_r$ we have
\[
\begin{split}
    \act_{t\inv} \big( F(r)(a) \big) &= \piact \big( F(r)(a) \big)(t) = \lim_i (\varphi \cdot k_i)^r (t) = \lim_i (\varphi \cdot k_i) (t , r\inv t) \\
        &= \varphi(r\inv t , t) \big( k_i(t , r\inv t) \big) = \varphi(r\inv t , t) \big( \act_{t\inv}(a) \big) .
\end{split}
\]
The claimed identity follows. 


The statement about module maps is an easy calculation using (\ref{eq:moduleactioninteractswithaction}).
\end{proof}

The following result characterises the Herz--Schur multipliers of $\act$ among the Schur multipliers of $\dcoact$. 
We identify $\dact[\dcoact]$ with the action $\act \otimes \Ad \rreg$ as in (\ref{eq:secdualactidentif}).

\begin{theorem}\label{th:HSmultsinSchmults}
Let $\act$ be an action of $\grp$ on $\vNal$ and $R$ a Schur multiplier on $\vNal \btens \Bd[\Ltwo]$.
The following are equivalent:
\begin{enumerate}[i.]
    \item $\piact[{\dact[\dcoact]}] \circ R = (R \otimes \id) \circ \piact[{\dact[\dcoact]}]$; 
    \item $R = \overline{S}$ for some Herz--Schur multiplier $S$ of $\act$.
\end{enumerate}
Moreover, if $\vNal$ has an $\module$-module structure satisfying (\ref{eq:moduleactioninteractswithaction}) then $R$ is an $\module$-module map if and only if $S$ is.
\end{theorem}
\begin{proof}
(i)$\implies$(ii) Since $R$ commutes with $\dact[\dcoact]$ we deduce that $R$ defines a map on the fixed points of this action, which can be identified with $\vNcrs$ (see {\it e.g.}\ \cite[Theorem II.1.1]{NaTa79}).
Observe that $(\Psi \otimes \id) \circ \piact[{\dcoact[{\beta}]}] \circ \Psi\inv$ restricts to the coaction $\piact[\dcoact]$ of $\grp$ on $\vNcrs$. 
Now calculate, using $\Psi \tilde{R} = R \Psi$ and Lemma~\ref{le:SchurmultgivesHSmultofdual},
\[
\begin{split}
    \piact[\dcoact] \circ R &= (\Psi \otimes \id) \circ \piact[{\dcoact[{\beta}]}] \circ \Psi\inv \circ R = (\Psi \otimes \id) \circ \piact[{\dcoact[{\beta}]}] \circ \tilde{R} \circ \Psi\inv \\
        &= (\Psi \otimes \id) \circ (\tilde{R} \otimes \id) \circ \piact[{\dcoact[{\beta}]}] \circ \Psi\inv = (R \otimes \id) \circ (\Psi \otimes \id) \circ \piact[{\dcoact[{\beta}]}] \circ \Psi\inv \\
        &= (R \otimes \id) \circ \piact[\dcoact] .
\end{split}
\]
Hence the restriction of $R$ to $\vNcrs$ is a Herz--Schur multiplier.

(ii)$\implies$(i) Suppose $S$ is a Herz--Schur multiplier of $\act$ with symbol $F$. Then, for any $a \in \vNal,\ r \in \grp,\ \phi \in \Linf$, using the equivalent of (\ref{eq:secdualactidentif}) for dual actions \cite[Theorem 2.7]{NaTa79},
\begin{equation*}\label{eq:HSisinvariantSchurcalc}
\begin{split}
    (\overline{S} \otimes \id) \circ \piact[{\dact[\dcoact]}] \big( \piact(a) \lrega_r (I_{\HH_\vNal} \otimes \phi) \big) &= (\overline{S} \otimes \id) \Big( \big( \piact(a) \lrega_r \otimes I_{L^2(\grp)} \big) \big( I_{\HH_\vNal} \otimes \piact[\actG](\phi) \big) \Big) \\
        &= \Big( \piact \big( F(r)(a) \big) \lrega_r \otimes I_{L^2(\grp)} \Big) \big( I_{\HH_\vNal} \otimes \piact[\actG](\phi) \big) \\
        &= \piact[{\dact[\dcoact]}] \left( \Big( \piact \big( F(r)(a) \big) \lrega_r \Big) ( I_{\HH_\vNal} \otimes \phi ) \right) \\
        &= \piact[{\dact[\dcoact]}] \circ \overline{S} \big( \piact(a) \lrega_r (I_{\HH_\vNal} \otimes \phi) \big) ,
\end{split}
\end{equation*}
so the claim follows by linearity and continuity.

If $S$ is a module map then $\overline{S}$ is also a module map by Proposition~\ref{pr:HSmultsareSmultsofdualcoact}.
On the other hand, if $\overline{S}$ is a module map then
\[
\begin{split}
    S \big( b \cdot \piact(a) \lrega_r \big) = \overline{S} \big( (b \otimes I_{L^2(\grp)}) \cdot \piact(a) \lrega_r \big) &= (b \otimes I_{L^2(\grp)}) \cdot \overline{S} \big( \piact(a) \lrega_r \big) \\ 
        &= b \cdot S \big( \piact(a) \lrega_r \big) ,
\end{split}
\]
so $S$ is also a module map.
\end{proof}

\begin{remark}\label{re:classicaltransference}
{\rm 
When $\vNal = \CC$ and $\act$ is trivial the above results recover the known fact \cite{BF84} that a Schur multiplier $S$ on $\Bd[\Ltwo]$ restricts to a Herz--Schur multiplier on $\vN$ if and only if $S$ commutes with the action $\Ad \rho$ (the second dual of the trivial action).
In this classical case Lemma~\ref{le:SchurmultgivesHSmultofdual} states that every Schur multiplier of $\Bd[\Ltwo]$ can be identified with a Herz--Schur multiplier of $\actG$. 
}
\end{remark}

In this section we have been careful to keep track of multipliers which respect an additional module structure. 
The reason is that the Herz--Schur multipliers of a semidirect product $H \rtimes G$ have an obvious identification with Herz--Schur multipliers of $\vNcros{\vN[H]}{G}{}$, and become $\falg[H]$-module maps under this identification.
In the next section we will make use of multipliers respecting this extra module structure.

\section{Abelian Groups}
\label{sec:abeliangrooups}

\noindent 
We now assume that $\grp$ is abelian, with dual group $\dgrp$.
By Takai duality $\Bd[\Ltwo[\grp]]$ is isomorphic to the crossed product formed by the coaction $\coactG$ dual to the trivial action of $\grp$ on $\CC$, or the action $\actG$ dual to the trivial action of $\dgrp$ on $\CC$.
For a map $S$ on $\Bd[\Ltwo]$ we write $S^{\actG}$ for the corresponding map on $\vNcros{\vN[\dgrp]}{\grp}{\actG}$ and $S^{\coactG}$ for the corresponding map on $\vNcros{\vN}{\dgrp}{\coactG}$, so for example $S^{\coactG} = \Phi \circ S \circ \Phi\inv$.

In \cite[Section 6]{MTT18} we raised the question of how the Herz--Schur multipliers of $\grp$ acting on $\vN[\dgrp]$ are related to $\FSalg \odot \FSalg[\dgrp]$; note that the convolution multipliers considered there are precisely those appearing in (i) below. 

\begin{theorem}\label{th:FSalgmodulemultipliers}
Let $S$ be a completely bounded, weak*-continuous map on $\Bd[\Ltwo]$.
The following are equivalent:
\begin{enumerate}[i.]
    \item $S^{\actG}$ is a Herz--Schur multiplier of $\actG$ and is an $\falg[\dgrp]$-module map;
    \item $S^{\coactG}$ is a Herz--Schur multiplier of $\coactG$ and is an $\falg$-module map.
\end{enumerate}
Moreover, the set of all $S$ satisfying the equivalent conditions can be identified with the space $\FSalg[\dgrp \times \grp]$. 
\end{theorem}
\begin{proof}
Observe that under the identifications of each crossed product with $\Bd[\Ltwo]$ the module action $\cdot$ of $\falg$ on $\vNcros{\vN}{\dgrp}{\coactG}$ (see (\ref{eq:modactionliftedtocross})) is carried to the action $*$ on $\vNcros{\vN[\dgrp]}{\grp}{\actG}$ of Remark~\ref{re:HSmultnotes}(iv), and the corresponding statement holds for the module actions of $\falg[\dgrp]$.
Assume (i) holds, take $u \in \falg$, $r \in \grp$ and $\gamma \in \dgrp$, and write $\mul_\gamma$ for the multiplication operator on $\Ltwo$ associated to $\gamma \in \dgrp$. Then
\[
\begin{split}
    S^{\coactG} \big( u \cdot \picoact[\coactG](\lreg_r) \mul_\gamma  \big) &= S \big( u(r) \lreg_r \mul_\gamma \big) = \overline{\dualp{\gamma}{r}} S \big( u(r) \mul_\gamma \lreg_r \big) \\
        &= \overline{\dualp{\gamma}{r}} S^{\actG} \big( u * \piact[\actG](\mul_\gamma) \lrega_r \big) = \overline{\dualp{\gamma}{r}} u * S^{\actG} \big( \piact[\actG](\mul_\gamma) \lrega_r \big) \\
        &= \overline{\dualp{\gamma}{r}} \dualp{\gamma}{r} u \cdot S^{\coactG} \big( \picoact[\coactG](\lreg_r) \mul_\gamma \big) = u \cdot S^{\coactG} \big( \picoact[\coactG](\lreg_r) \mul_\gamma \big) ,
\end{split}
\]
so $S$ defines an $\falg$-module map on $\vNcros{\vN}{\dgrp}{\coactG}$.
Similarly we calculate that $S^{\coactG}(v * \picoact[\coactG](\lreg_r) \mul_\gamma ) = v * S^{\coactG} ( \picoact[\coactG](\lreg_r) \mul_\gamma )$ for each $v \in \falg[\dgrp]$, so $S^{\coactG}$ is a Herz--Schur multiplier of the action $\dcoact$ by Remark~\ref{re:HSmultnotes}(iii).
We have now shown that (i) implies (ii); by Pontryagin duality the same proof shows (ii) implies (i).


Now let $S$ satisfy (i); if $F$ denotes the symbol of $S^{\actG}$ then, for any $r \in \grp,\ \gamma \in \dgrp$, $F(r)$ is an $\falg[\dgrp]$-module map, and therefore a Herz--Schur multiplier of $\dgrp$, so we identify $F$ with a map $\dgrp \times \grp \to \CC$.
Consider the Schur multiplier $\overline{S^{\actG}}$; it will be convenient to regard $\overline{S^{\actG}}$ as acting on $\Linf \btens \Bd[\Ltwo[\grp]]$. 
The restriction of $\overline{S^{\actG}}$ to $\Linf \btens \vN$ is a completely bounded, weak*-continuous map; to see that it preserves $\Linf \btens \vN$ we calculate, using the modularity of $\overline{S^{\actG}}$,
\[
\begin{split}
    \overline{S^{\actG}} (\mul_\gamma \otimes \lreg_r) &= (1 \otimes \mul_{\gamma\inv}) \overline{S^{\actG}} (\mul_\gamma \otimes \mul_\gamma \lreg_r) \\
        &= (1 \otimes \mul_{\gamma\inv}) \Phi\inv \big( (S^{\actG} \otimes \id) (\mul_\gamma \otimes \mul_\gamma \lreg_r \otimes \lreg_r ) \big) \\
        &= (1 \otimes \mul_{\gamma\inv}) \Phi\inv \big( F(\gamma , r) \mul_\gamma \otimes \mul_\gamma \lreg_r \otimes \lreg_r) \big) \\
        &= F(\gamma , r) (\mul_\gamma \otimes \lreg_r) .
\end{split}
\]
From this calculation we also see that this restriction is an $\falg[\dgrp \times \grp]$-module map on $\Linf \btens \vN$, and therefore a Herz--Schur multiplier.

Conversely, if $S \in \FSalg[\dgrp \times \grp]$, with symbol $u : \dgrp \times \grp \to \CC$, consider the associated Schur multiplier $\overline{S}$ acting on $\Bd[\Ltwo[\grp] \otimes \Ltwo]$.
The restriction of $\overline{S}$ to $\vNcros{\Linf}{\grp}{\actG}$ is a Herz--Schur multiplier of $\actG$, since
\[
    \overline{S} \left( \piact[\actG](\mul_\gamma) \lrega_r \right) = \overline{S} (\mul_\gamma \otimes \mul_\gamma \lreg_r) = (1 \otimes \mul_\gamma) \overline{S} (\mul_\gamma \otimes \lreg_r) = u(\gamma ,  r) \piact[\actG](\mul_\gamma) \lrega_r .
\]
That $\overline{S}$ commutes with the $\falg[\dgrp]$-module action also follows easily.
\end{proof}


\bibliography{multipliersduality}

\begin{bibdiv}
\begin{biblist}

\bib{AD87}{article}{
      author={Anantharaman-Delaroche, Claire},
       title={Syst\`emes dynamiques non commutatifs et moyennabilit\'e},
        date={1987},
        ISSN={0025-5831},
     journal={Math. Ann.},
      volume={279},
      number={2},
       pages={297\ndash 315},
         url={http://dx.doi.org/10.1007/BF01461725},
      review={\MR{919508}},
}

\bib{BC15}{article}{
      author={B{\'e}dos, Erik},
      author={Conti, Roberto},
       title={Fourier series and twisted {${\rm C}^{\ast}$}-crossed products},
        date={2015},
        ISSN={1069-5869},
     journal={J. Fourier Anal. Appl.},
      volume={21},
      number={1},
       pages={32\ndash 75},
      review={\MR{3302101}},
}

\bib{BGB18}{unpublished}{
      author={Blasco, Oscar},
      author={Garc\'ia-Bayona, Ismael},
       title={Schur product with operator valued entries},
        date={2019},
        note={Preprint. arXiv: 1804.03432 [math.FA]},
}

\bib{BF84}{article}{
      author={Bo{\.z}ejko, Marek},
      author={Fendler, Gero},
       title={Herz-{S}chur multipliers and completely bounded multipliers of
  the {F}ourier algebra of a locally compact group},
        date={1984},
     journal={Boll. Un. Mat. Ital. A (6)},
      volume={3},
      number={2},
       pages={297\ndash 302},
      review={\MR{753889}},
}

\bib{Bra17}{incollection}{
      author={Brannan, Michael},
       title={Approximation properties for locally compact quantum groups},
        date={2017},
   booktitle={Topological quantum groups},
      series={Banach Center Publ.},
      volume={111},
   publisher={Polish Acad. Sci. Inst. Math., Warsaw},
       pages={185\ndash 232},
      review={\MR{3675051}},
}

\bib{BrO08}{book}{
      author={Brown, Nathanial~P.},
      author={Ozawa, Narutaka},
       title={{$C^*$}-algebras and finite-dimensional approximations},
      series={Graduate Studies in Mathematics},
   publisher={American Mathematical Society, Providence, RI},
        date={2008},
      volume={88},
        ISBN={978-0-8218-4381-9; 0-8218-4381-8},
         url={https://doi.org/10.1090/gsm/088},
      review={\MR{2391387}},
}

\bib{dCH85}{article}{
      author={De~Canni\`ere, Jean},
      author={Haagerup, Uffe},
       title={Multipliers of the {F}ourier algebras of some simple {L}ie groups
  and their discrete subgroups},
        date={1985},
        ISSN={0002-9327},
     journal={Amer. J. Math.},
      volume={107},
      number={2},
       pages={455\ndash 500},
         url={https://doi.org/10.2307/2374423},
      review={\MR{784292}},
}

\bib{DR12}{article}{
      author={Dong, Zhe},
      author={Ruan, Zhong-Jin},
       title={A {H}ilbert module approach to the {H}aagerup property},
        date={2012},
        ISSN={0378-620X},
     journal={Integral Equations Operator Theory},
      volume={73},
      number={3},
       pages={431\ndash 454},
         url={https://doi.org/10.1007/s00020-012-1979-3},
      review={\MR{2945214}},
}

\bib{EfRu00}{book}{
      author={Effros, Edward~G.},
      author={Ruan, Zhong-Jin},
       title={Operator spaces},
      series={London Mathematical Society Monographs. New Series},
   publisher={The Clarendon Press, Oxford University Press, New York},
        date={2000},
      volume={23},
        ISBN={0-19-853482-5},
      review={\MR{1793753}},
}

\bib{Gr53}{article}{
      author={Grothendieck, Alexander},
       title={R\'esum\'e de la th\'eorie m\'etrique des produits tensoriels
  topologiques},
        date={1996},
        ISSN={0104-3854},
     journal={Resenhas},
      volume={2},
      number={4},
       pages={401\ndash 480},
        note={Reprint of Bol. Soc. Mat. S{\~a}o Paulo {{\bf{8}}} (1953), 1--79
  [ MR0094682 (20 \#1194)]},
      review={\MR{1466414}},
}

\bib{Haaun}{unpublished}{
      author={Haagerup, Uffe},
       title={Decomposition of completely bounded maps on operator algebras},
        note={Unpublished manuscript.},
}

\bib{Haa78}{article}{
      author={Haagerup, Uffe},
       title={An example of a nonnuclear {$C^{\ast} $}-algebra, which has the
  metric approximation property},
        date={1978/79},
     journal={Invent. Math.},
      volume={50},
      number={3},
       pages={279\ndash 293},
}

\bib{Her74}{article}{
      author={Herz, Carl},
       title={Une g\'en\'eralisation de la notion de transform\'ee de
  {F}ourier-{S}tieltjes},
        date={1974},
        ISSN={0373-0956},
     journal={Ann. Inst. Fourier (Grenoble)},
      volume={24},
      number={3},
       pages={xiii, 145\ndash 157},
      review={\MR{0425511}},
}

\bib{Jo92}{article}{
      author={Jolissaint, Paul},
       title={A characterization of completely bounded multipliers of {F}ourier
  algebras},
        date={1992},
        ISSN={0010-1354},
     journal={Colloq. Math.},
      volume={63},
      number={2},
       pages={311\ndash 313},
      review={\MR{1180643}},
}

\bib{JNR09}{article}{
      author={Junge, Marius},
      author={Neufang, Matthias},
      author={Ruan, Zhong-Jin},
       title={A representation theorem for locally compact quantum groups},
        date={2009},
        ISSN={0129-167X},
     journal={Internat. J. Math.},
      volume={20},
      number={3},
       pages={377\ndash 400},
         url={https://doi.org/10.1142/S0129167X09005285},
      review={\MR{2500076}},
}

\bib{KaPa05}{article}{
      author={Katavolos, Aristides},
      author={Paulsen, Vern~I.},
       title={On the ranges of bimodule projections},
        date={2005},
        ISSN={0008-4395},
     journal={Canad. Math. Bull.},
      volume={48},
      number={1},
       pages={97\ndash 111},
         url={http://dx.doi.org/10.4153/CMB-2005-009-4},
      review={\MR{2118767}},
}

\bib{MTT18}{article}{
      author={McKee, Andrew},
      author={Todorov, Ivan~G.},
      author={Turowska, Lyudmila},
       title={Herz--{S}chur multipliers of dynamical systems},
        date={2018},
        ISSN={0001-8708},
     journal={Adv. Math.},
      volume={331},
       pages={387\ndash 438},
         url={https://doi.org/10.1016/j.aim.2018.04.002},
      review={\MR{3804681}},
}

\bib{NaTa79}{book}{
      author={Nakagami, Yoshiomi},
      author={Takesaki, Masamichi},
       title={Duality for crossed products of von {N}eumann algebras},
      series={Lecture Notes in Mathematics},
   publisher={Springer, Berlin},
        date={1979},
      volume={731},
        ISBN={3-540-09522-5},
      review={\MR{546058}},
}

\bib{Pis96}{book}{
      author={Pisier, Gilles},
       title={Similarity problems and completely bounded maps},
      series={Lecture Notes in Mathematics},
   publisher={Springer-Verlag, Berlin},
        date={1996},
      volume={1618},
        ISBN={3-540-60322-0},
         url={http://dx.doi.org/10.1007/978-3-662-21537-1},
      review={\MR{1441076}},
}

\bib{PSS00}{article}{
      author={Pop, Florin},
      author={Sinclair, Allan~M.},
      author={Smith, Roger~R.},
       title={Norming {$C^\ast$}-algebras by {$C^\ast$}-subalgebras},
        date={2000},
        ISSN={0022-1236},
     journal={J. Funct. Anal.},
      volume={175},
      number={1},
       pages={168\ndash 196},
         url={https://doi.org/10.1006/jfan.2000.3601},
      review={\MR{1774855}},
}

\bib{Tak02}{book}{
      author={Takesaki, Masamichi},
       title={Theory of operator algebras. {I}},
      series={Encyclopaedia of Mathematical Sciences},
   publisher={Springer-Verlag, Berlin},
        date={2002},
      volume={124},
        ISBN={3-540-42248-X},
        note={Reprint of the first (1979) edition, Operator Algebras and
  Non-commutative Geometry, 5},
      review={\MR{1873025}},
}

\end{biblist}
\end{bibdiv}

\end{document}